\documentclass[11pt,twoside]{amsart}

\usepackage[latin1]  {inputenc}%
\usepackage[T1]      {fontenc }%
\usepackage          {amsmath }%
\usepackage          {amsfonts}%
\usepackage          {amssymb }%
\usepackage          {amsthm  }%
\usepackage          {a4wide  }%
\usepackage          {url     }%
\usepackage          {tikz    }%
\usepackage[bookmarks=false,pdfborder={0 0 0.05}]{hyperref}

\DeclareMathOperator{\Red}{Red}

\def\defn#1{{\sf #1}}

\newcommand{\edge}{\mathbin{\tikz [semithick, baseline=-0.2ex,-latex, ->] \draw [-] (0pt,0.4ex) -- (1em,0.4ex);}} % edge
\newcommand{\edgedir}{\mathbin{\tikz [semithick, baseline=-0.2ex,-latex, ->] \draw [->] (0pt,0.4ex) -- (1em,0.4ex);}} % edge
\newcommand{\edgedirback}{\mathbin{\tikz [semithick, baseline=-0.2ex,-latex, ->] \draw [<-] (0pt,0.4ex) -- (1em,0.4ex);}} % edge
\newcommand{\Bruhat}{\Omega}
\newcommand{\Bruhatdir}{\Omega_{\operatorname{dir}}}

\newtheorem{theorem}{Theorem}[section]

\newtheorem{proposition}[theorem]{Proposition}
\newtheorem{lemma}[theorem]{Lemma}

\theoremstyle{definition}

\newtheorem{remark}[theorem]{Remark}

\title[On the transitive Hurwitz action on decompositions of Coxeter elements]{A note on the transitive Hurwitz action on\\decompositions of parabolic Coxeter elements}

\author[B.~Baumeister]{Barbara Baumeister}
\address{Barbara Baumeister, Universit\"at Bielefeld, Germany}
\email{b.baumeister@math.uni-bielefeld.de}
\author[M.~Dyer]{Matthew Dyer}
\address{Matthew Dyer, University of Notre Dame, United States}
\email{dyer.1@nd.edu}
\author[C.~Stump]{Christian Stump}
\address{Christian Stump, Freie Universit\"at Berlin, Germany}
\email{christian.stump@fu-berlin.de}
\author[P.~Wegener]{Patrick Wegener}
\address{Patrick Wegener, Universit\"at Bielefeld, Germany}
\email{pwegener@math.uni-bielefeld.de}

\date{\today}

\begin{document}

\begin{abstract}
In this note, we provide a short and self-contained proof that the braid group 
on~$n$ strands acts transitively on the set of reduced factorizations of a 
Coxeter element in a Coxeter group of finite rank~$n$ into products of 
reflections.
We moreover use the same argument to also show that all factorizations of an 
element in a parabolic subgroup of~$W$ lie as well in this parabolic subgroup.
\end{abstract}

\maketitle

\section{Introduction}
Let $(W,T)$ be a \defn{dual Coxeter system} of finite rank $m$ in the sense 
of~\cite{Be03}.
This is to say that there is a subset $S \subseteq T$ with $|S| = m$ such that 
$(W,S)$ is a Coxeter system, and $T = \big\{ wsw^{-1} : w \in W, s \in S \big\}$ 
is the set of reflections for the Coxeter system $(W,S)$.
We then call~$(W,S)$ a \defn{simple system} for $(W,T)$.
Such simple systems for $(W,T)$ were studied by several authors, see 
e.g.~\cite{FHM06} and the references therein.
In particular, if $S$ is a simple system for $(W,T)$ then so is $wSw^{-1}$ for 
any $w \in W$.
It is moreover shown in~\cite{FHM06} that for important classes, all simple 
systems for $(W,T)$ are conjugate to one another in this sense.

For $w \in W$, we furthermore denote by $\Red_T(w)$ the set of reduced 
factorizations of~$W$ into reflections.
An element $c \in W$ is called a \defn{parabolic Coxeter element} for $(W,T)$ 
if there is a simple system~$S = \{ s_1,\ldots,s_m \}$ such that 
$c = s_1\cdots s_n$ for some $n \leq m$.
Similarly, we call the reflection subgroup generated by $\{s_1,\ldots,s_n\}$ a \defn{parabolic subgroup}.
The element~$c$ is moreover called a \defn{standard parabolic Coxeter 
element} for the Coxeter system $(W,S)$.

\begin{remark}
  Observe that this definition of parabolic Coxeter elements is more general 
  than usual.
  The simplest Coxeter group for which this definition is indeed more general 
  than considering conjugates of a fixed simple system is the finite Coxeter 
  group of type $H_2 = I_2(5)$ given by all linear transformations of the plane that leave a regular pentagon invariant.
  One choice for a simple system is given by two reflections through two 
  consecutive vertices of the pentagon, another choice is the product of two 
  reflections through two vertices with distance two.
  Both choices of simple systems generate the same set of reflections, even 
  though both are not conjugate.
\end{remark}

The \defn{braid group} on $n$ strands is the group $B_n$ with generators 
$\sigma_1,\ldots , \sigma_{n-1}$ subject to the relations
\begin{align*}
 \sigma_i \sigma_j & = \sigma_j \sigma_i \quad \text{for } |i-j| > 1,\\
\sigma_i \sigma_{i+1} \sigma_i & = \sigma_{i+1} \sigma_i \sigma_{i+1}.
\end{align*}
It acts on the set $T^n$ of $n$-tuples of reflections as
\begin{align*}
\sigma_i (t_1 ,\ldots , t_n ) &= (t_1 ,\ldots , t_{i-1} , \hspace*{5pt} t_i t_{i+1} t_i,
\hspace*{5pt} \phantom{t_{i+1}}t_i\phantom{t_{i+1}}, \hspace*{5pt} t_{i+2} ,
\ldots , t_n), \\
\sigma_i^{-1} (t_1 ,\ldots , t_n ) &= (t_1 ,\ldots , t_{i-1} , \hspace*{5pt} \phantom
{t_i}t_{i+1}\phantom{t_i}, \hspace*{5pt} t_{i+1}t_it_{i+1}, \hspace*{5pt} t_{i+2} ,
\ldots , t_n).
\end{align*}
For example, if $n=2$, then the action of $\sigma_{1}$ is described by 
\begin{equation*}\ldots\mapsto(srs,srsrs)\mapsto (s,srs)\mapsto (r,s)\mapsto
(rsr,r)\mapsto(rsrsr,rsr)\mapsto\ldots\end{equation*}
for any $r,s\in T$. Note that in this case, the $B_{2}$-orbit of $(r,s)$ is the set 
of all pairs  $(t_{1},t_{2})$ of reflections of the subgroup $\langle r,s\rangle$, 
such that $t_{1}t_{2}=rs$. 

\medskip

The following lemma is a direct consequence of the definition.
\begin{lemma}
\label{lem:HurwitzOnFactorizations}
  Let $W'$ be a reflection subgroup of~$W$ and let $T' = T \cap W'$ be the 
  set of reflections in~$W'$.
  Then the braid group on~$n$ strands acts on $\Red_{T'}(w)$ for $w \in W'$.
\end{lemma}
\begin{proof}
  Given a reduced factorization $w = t_1,\ldots,t_n$, and let 
  $\sigma_i(t_1,\ldots,t_n) = (t'_1,\ldots,t'_n)$.
  The lemma then follows from the two observations that 
  $t_1 \cdots t_n = t'_1 \cdots t'_n$ and $\{ t_1,\ldots,t_n \} \subseteq T'$ if and 
  only if $\{ t'_1,\ldots,t'_n \} \subseteq T'$.
\end{proof}
This action on $\Red_{T'}(w)$ is also known as the \defn{Hurwitz action}.
For finite Coxeter systems, the Hurwitz action was first shown to act 
transitively on $\Red_T(c)$ for a Coxeter element~$c$ in a letter from 
P.~Deligne to E.~Looijenga~\cite{De74}. The first published proof is due to 
D.~Bessis and can be found in~\cite{Be03}.
K.~Igusa and R.~Schiffler generalized this result to arbitrary Coxeter groups 
of finite rank; see~\cite[Theorem~1.4]{IS10}.
This transitivity has important applications in the theory of Artin groups, 
see~\cite{Be03,Di06}, and as well as in the representation theory of algebras; 
see~\cite{IS10,Ig11,HK13}.
\medskip

The aim of this note is to provide a simple proof of K.~Igusa and R.~Schiffler's 
theorem, based on arguments similar 
to those in~\cite{Dy01}.
We moreover emphasize that the condition on the Coxeter element~$c \in W$ 
in this note is slightly relaxed from the condition in the original theorem; 
compare~\cite[Theorem~1.4]{IS10}.

\begin{theorem}
\label{th:maintheorem1}
  Let $(W,T)$ be a dual Coxeter system of finite rank~$m$ and 
  let~$c = s_1\cdots s_n$ be a parabolic Coxeter element in~$W$.
  The Hurwitz action on $\Red_T(c)$ is transitive.
  In symbols, for each $(t_1,\ldots,t_n) \in T^n$ such that 
  $c = t_1 \cdots t_n$, there is a braid group element $b \in B_n$ such that
  $$b(t_1,\ldots,t_n) = (s_1,\ldots,s_n).$$
\end{theorem}
By the observation in Lemma~\ref{lem:HurwitzOnFactorizations}, this theorem 
has the direct consequence that the parabolic subgroup 
$\langle s_1,\ldots,s_n\rangle \trianglelefteq W$ 
does indeed not depend on the particular factorization $c = s_1 \cdots s_n$ 
but only on the parabolic Coxeter element~$c$ itself.
We thus denote this parabolic by $W_c := \langle t_1,\ldots,t_n \rangle$ for 
any factorization $c = t_1 \cdots t_n$.
We moreover obtain that $\Red_T(c) = \Red_{T'}(c)$ with $T' = W_c \cap T$ 
being the set of reflections in the parabolic subgroup $W_c$.
The main argument in the proof of this theorem 
(see Proposition~\ref{prop:directedpath} below) will also imply the following 
theorem that extends this direct consequence to all elements in
a parabolic subgroup.
\begin{theorem}
\label{th:maintheorem2} Let $W'$ be a parabolic subgroup of $W$. 
Then for any $w\in W'$, 
  $$\Red_T(w) = \Red_{T'}(w),$$
  where $T' = W' \cap T$ is the set of reflections in $W'$.
\end{theorem}

\section{The proof}
For the  proof of the two theorems, we fix a Coxeter system $(W,S)$.
Denote by~$\ell = \ell_S$ and by~$\ell_T$ the length function on~$W$ with 
respect to the simple generators~$S$ and with respect to the generating 
set~$T$, respectively.
Since $S \subseteq T$, we have that $\ell_T(w) \leq \ell(w)$ for all~$w \in W$.

\medskip

The following lemma provides an alternative description of standard 
parabolic Coxeter elements.

\begin{lemma}\label{lemma:factorizationlength}
  An element $w \in W$ is a standard parabolic Coxeter element for 
  $(W,S)$ if and only if $\ell_T(w) = \ell(w)$.
\end{lemma}
\begin{proof}
  Given a reduced expression $w = s_{i_1} \ldots s_{i_k}$, it was shown 
  in~\cite[Theorem~1.1]{Dy01} that $\ell_T(w)$ is given by the minimal 
  number of simple generators that can be removed from 
  $s_{i_1} \ldots s_{i_k}$ to obtain the identity.
  This yields that $\ell_T(w) = k = \ell(w)$ if and only if 
  $s_{i_1} \ldots s_{i_k}$ does not contain any generator twice.
\end{proof}
Define  the \defn{Bruhat graph} $\Bruhat$ for the dual Coxeter system 
$(W,T)$ as the undirected graph on vertex set $W$ with edges given by 
$w \edge wt$ for $t \in T$.
For any   factorization $w = t_1 \cdots t_n \in W$ with $t_{i}\in T$  and  any 
$x\in W$, there is  a corresponding path 
$$x \edge xt_1 \edge xt_1t_2 \edge \ldots \edge xt_1\cdots t_n = xw$$ 
from $x$ to $xw$
in~$\Bruhat$. 
It is clear that  the factorization of $w$ is reduced  if and only if the 
corresponding path from $x$ to $xw$ has minimal length among paths from 
$x$ to $xw$ for some (equivalently, every) $x\in W$.
The  simple system $(W,S)$ induces an orientation on $\Bruhat$ given by 
$w \edgedir wt$ if $\ell(w) < \ell(wt)$.
We denote the resulting \defn{directed Bruhat graph} by $\Bruhatdir$.

\bigskip

The proof of the two main results is based on the case $x=e$ of the  following 
proposition.
\begin{proposition}
\label{prop:directedpath}
  Let $(W,S)$ be a Coxeter system.
  Moreover, let $w = t_1 \cdots t_n \in W$ be a reduced factorization of an 
  element in~$W$ into reflections, and let
  $$x \edge xt_1 \edge xt_1t_2 \edge \ldots \edge xt_1\cdots t_n = xw$$
  be the corresponding path in $\Bruhat$ starting at an element $x \in W$.
  Then there is a factorization $w = t'_1 \cdots t'_n$ in the Hurwitz orbit of the 
  factorization $t_1 \cdots t_n$ such that the corresponding path in $\Bruhatdir$
  starting at $x$ is first decreasing in length, then increasing; more precisely, it 
   is of the form 
  $$ x\edgedirback xt'_1 \edgedirback xt'_1t'_2\edgedirback \ldots  
  \edgedirback xt_{1}'\cdots t_{i}'\edgedir  \ldots \edgedir xt'_1
   \cdots t'_n = xw$$ for some (unique) integer  $i$ with $0\leq i\leq n$.
  In the special case $x=e$, this gives a directed  path 
  $$ e\edgedir t'_1 \edgedir t'_1t'_2 \edgedir \ldots 
  \edgedir t'_1 \cdots t'_n = w$$
    in $\Bruhatdir$.
\end{proposition}
\begin{proof}
  First consider two distinct reflections $t_1$ and $t_2$ and an element 
  $z \in W$ such that
   $z \edgedir zt_1 \edgedirback zt_1t_2$ in $\Bruhatdir$.
  We claim that  there exist reflections $t'_1,t'_2 \in \langle t_1,t_2 \rangle$ 
  with $t_1t_2 = t'_1t'_2$ such that 
  $z \edgedir zt'_1 \edgedir zt'_1t'_2$ or $z \edgedirback zt'_1
  \edgedirback zt_{1}'t_{2}'$
  or $z \edgedirback zt'_1\edgedir zt_{1}'t_{2}'$.
    This implies, by the comment before Lemma 
    \ref{lem:HurwitzOnFactorizations}, that  one can get from the factorization 
    $t_1 t_2$ to the factorization $t'_1 t'_2$ inside $W' = \langle t_1,t_2 \rangle$ 
    by braid moves, and hence in particular  that $W' = \langle t'_1,t'_2 \rangle$.
    Moreover, one has  $\ell(zt_{1}')< \max(\ell(z),\ell(zt_{1}t_{2}))<\ell(zt_{1})$.

  To prove the claim, consider the coset $z\langle t_1,t_2 \rangle$ in $W$. 
  By~\cite[Theorem~2.1]{Dy01}, the proof immediately reduces to the case 
  $W' = \langle t_1,t_2 \rangle$ dihedral.
  We  check this case directly.
  To this end, let $s'_1,s'_2$ be the Coxeter generators of $W'$, and observe 
  that any reflection (element of odd length) and any rotation (element of even 
  length) in $W'$ are joined by an edge in $\Bruhat$, which in $\Bruhatdir$ is 
  oriented towards the element of greater length  with respect  to the generating
   set $s'_1,s'_2$.
  Therefore given $z \in W'$, there are three situations: either 
  $\ell(z) < \ell(zt_1t_2)$, or $\ell(z) > \ell(zt_1t_2)$, or $\ell(z) = \ell(zt_1t_2)$.
 This implies that one can choose $t'_1$ and $t'_2$ with $t'_1 t'_2 = t_1t_2$ in 
 the three situations such that $z \edgedir zt'_1 \edgedir zt'_1t'_2$,  
 $z \edgedirback zt'_1 \edgedirback zt'_1t'_2$ or 
  $z \edgedirback zt'_1 \edgedir zt'_1t'_2$ respectively (note that in the third 
  case, one has  $z\neq e$ or else $l(zt_{1}t_{2})=l(z)$ implies  $t_{1}=t_{2}$, 
  contrary to assumption).

Consider    the path  in $\Bruhat$  attached to $w=t_{1}\cdots t_{n}$ and 
beginning at $x$. Any subpath $z\edge zt_{1}\edge zt_{1}t_{2}$ as in the claim 
may be replaced by a path $z\edge zt'_{1}\edge zt'_{1}t'_{2}$ as there, to give 
a new path from $x$ to $xw$ of the same length $n$; we call this a 
``replacement.'' Apply to the original path  a sequence of successive 
replacements.  Any path so obtained corresponds to the path beginning at $x$
 attached to some $T$-reduced expression of $w$ in the same Hurwitz orbit 
 as $t_{1}\cdots t_{n}$, and is a shortest path in $\Bruhat$ from $x$ to $xw$. 
 Note that  a replacement  of any subpath $z \edgedir zt_{1}' 
 \edgedirback zt'_{1}t'_{2}$ of such a path  is possible since the path's minimal 
 length  implies that  $t_{1}'\neq t_{2}'$.   Each replacement decreases the total 
 sum of the $\ell$-lengths of the vertices of the path, so eventually one obtains 
 a path in which no further replacements  are possible i.e. of the desired 
 decreasing-then-increasing form.   Finally, if $x=e$, then $i=0$ since  there 
 are no  paths $e\edgedirback t$.
\end{proof}

Given this proposition, we are finally in the position to prove the two main 
results of this note.
\begin{proof}[Proof of Theorem~\ref{th:maintheorem2}]
  Consider a  Coxeter system $(W,S)$,  a reflection subgroup 
  $W' \trianglelefteq W$.  It is known that the directed Bruhat graph for $W'$ 
  corresponding to the simple system of $W'$ induced by $S$ is the full 
  subgraph  $\Bruhatdir\big|_{W'}$ of $\Bruhatdir$ on vertex set  $W'$; 
  see~\cite[Theorem~2.1]{Dy01}.
 
 Let $w \in W'$. Then 
 Lemma~\ref{lem:HurwitzOnFactorizations},
 ~Proposition~\ref{prop:directedpath} (with $x=e$), and the discussion before 
 Proposition \ref{prop:directedpath} imply that $\Red_T(w) = \Red_{T'}(w)$ if 
 and only if every shortest directed path from $e$ to $w$ in $\Bruhatdir$ lies  
 inside $\Bruhatdir\big|_{W'}$.   
 
 Now assume that $W'$ is the standard parabolic subgroup generated by 
 some subset of $S$.  Then it is well known that every $S$-reduced 
 expression for $w\in W'$ is actually inside $W'$.
  It therefore follows that in this situation any shortest directed path from
   $e$ to $w$ in $\Bruhatdir$ indeed lies  inside $\Bruhatdir\big|_{W'}$.
  The theorem  follows by the above equivalence.
\end{proof}

\begin{proof}[Proof of Theorem~\ref{th:maintheorem1}]
  Again, fix a parabolic Coxeter element~$c = s_1\cdots s_n \in W$ and a 
  corresponding simple system $(W,S)$, and denote by $\Bruhat$ and 
  $\Bruhatdir$ the undirected and directed version of the Bruhat graph for 
  $(W,S)$.
  By Proposition~\ref{prop:directedpath}, it is left to show that any two directed 
  paths from~$e$ to~$c$ in $\Bruhatdir$ are in the same Hurwitz orbit.
  Let therefore
  $$e \edgedir t_1 \edgedir t_1t_2 \edgedir \ldots \edgedir t_1\cdots t_n = c$$
  be such a path.
  We have seen in Lemma~\ref{lemma:factorizationlength} that 
  $\ell(c) = \ell_T(c) = n$.
  It thus follows that $\ell(t_1\cdots t_i) = \ell_T(t_1\cdots t_i) = i$ for any~$i$.
  The strong exchange condition, see e.g.~\cite[Theorem~5.8]{Hu90}, then 
  yields that $t_1 \cdots t_{i+1}$ is obtained from $t_1\cdots t_i$ by adding a 
  single simple generator.
  Therefore, such a path is (bijectively) encoded by a permutation 
  $\pi = [\pi_1,\ldots,\pi_n]$ in which the simple generators are inserted into the 
  sequence.
  But given the factorization corresponding to such a path, it is straightforward 
  to see that the embedding of the permutation into the braid group 
  (by sending a simple transposition $(i,i+1)$ to the generator 
  $\sigma_i$ of~$B_n$) yields a braid that turns the given factorization into the 
  factorization $s_1\cdots s_n$.
  To this end, observe that given two factorizations encoded by two permutations $\pi_1$ and $\pi_2$ with $\ell(\pi_1) = \ell(\pi_2)+1$ and such that these differ only by a single simple transposition $\pi_2^{-1}\pi_1 = (i,i+1)$ for some index~$i$.
  Then the given factorizations are obtained from each other by applying the braid group generator~$\sigma_i$ to the factorization corresponding to $\pi_1$ to obtain the factorization corresponding to $\pi_2$.
  As the factorization corresponding to the identity permutation $[1,\ldots,n]$ is $(s_1,\ldots,s_n)$, the claim follows.
  As an example, consider the path
  $$e \edgedir s_2 \edgedir s_2 s_5 \edgedir s_2 s_3 s_5 
  \edgedir s_1 s_2 s_3 s_5 \edgedir s_1 s_2 s_3 s_4 s_5 = c.$$
  The corresponding factorization of~$c$ is given by
  $$c = s_2 \cdot s_5 \cdot s_5 s_3 s_5 \cdot s_5 s_3 s_2 s_1 s_2 s_3 s_5 
  \cdot s_5 s_4 s_5,$$
  and the permutation is $\pi = [2,5,3,1,4] = (1,2)(2,3)(4,5)(3,4)(2,3)$.
  On the other hand,
  $$\sigma_1 \sigma_2 \sigma_4 \sigma_3 \sigma_2(s_2,s_5,s_5 s_3 s_5,s_5 
  s_3 s_2 s_1 s_2 s_3 s_5,s_5 s_4 s_5) = (s_1,s_2,s_3,s_4,s_5),$$
  as desired.
\end{proof}

\end{document}